\documentclass[12pt]{article}

\usepackage{e-jc}
\usepackage{amsthm,amsmath,amssymb}
\usepackage{amsfonts}
\usepackage{verbatim}
\usepackage{url}
\usepackage{graphicx}
\usepackage{epstopdf}

\newtheorem{theorem}{Theorem}

\newtheorem{corollary}[theorem]{Corollary}
\newtheorem{lemma}[theorem]{Lemma}
\newtheorem{proposition}[theorem]{Proposition}
\numberwithin{equation}{section}
\numberwithin{theorem}{subsection}

\title{Oriented Hypergraphs I: Introduction and Balance}

\author{Lucas J. Rusnak\thanks{A special thanks to Thomas Zaslavsky and Gerard Cornu\'{e}jols for their feedback.}\\
\small Department of Mathematics\\[-0.8ex]
\small Texas State University\\[-0.8ex] 
\small San Marcos, Texas, U.S.A.\\
\small\tt lucas.rusnak@txstate.edu\\
}

\date{\dateline{Sept 30, 2012}{MMM DD, YYYY}}

\begin{document}

\maketitle

\begin{abstract}
 An oriented hypergraph is an oriented incidence structure that extends the
concept of a signed graph. We introduce hypergraphic structures
and techniques central to the extension of the circuit classification of signed graphs 
to oriented hypergraphs. Oriented hypergraphs are further decomposed into three 
families -- balanced, balanceable, and unbalanceable -- and we obtain a complete 
classification of the balanced circuits of oriented hypergraphs.

\end{abstract}


\section{Introduction}
\label{Section1}

Oriented hypergraphs have recently appeared in \cite{AH1} as an extension of the signed graphic incidence, adjacency, and Laplacian matrices to examine walk counting. This paper further expands the theory of oriented hypergraphs by examining the extension of the cycle space of a graph to oriented hypergraphs, and we obtain a classification of the balanced minimally dependent columns of the incidence matrix of an oriented hypergraph.

It is known that the cycle space of a graph characterizes the dependencies of the 
graphic matroid and the minimal dependencies, or circuits, are the edge sets 
of the simple cycles of the graph. Oriented hypergraphs have a natural division into three categories: balanced, balanceable, and unbalanceable. The family of balanced oriented hypergraphs contain graphs, so a characterization of the balanced circuits of oriented hypergraphs can be regarded as an extension of the following theorem: 

\begin{theorem}
\label{G1} $C$ is the edge set of a circuit of a graph $G$ if, and only
if, $C$ is a circuit of the graphic matroid $M(G)$.
\end{theorem}

The development of hypergraphic incidence orientation is a direct extension of the work by
 Zaslavsky in \cite{SG, BG1, OSG}, while the concept of balance is a relaxed version of the concepts which appear in \cite{Berge1, Berge2, BM, TrAlpha}.

Section \ref{Section2} 
introduces basic oriented hypergraphic definitions and the incidence matrix. 
Section \ref{Section3} collects the operations relevant to the classification of 
oriented hypergraphic circuits.  Section \ref{Section4} discusses hypergraphic 
extensions of paths and cycles.  Section \ref{Section5} introduces the 
hypergraphic cyclomatic number and the incidence graph. 
These lead to the development of the concept of balance for oriented 
hypergraphs in Section \ref{Section6}, and a complete classification of the 
balanced oriented hypergraphic circuits in Section \ref{Section7}.

\subsection{Signed Graphs}

A signed graph is a generalization of a graph that allows edges incident to $2$ or 
fewer vertices and signs every $2$-edge $+$ or $-$. The edges
incident to zero vertices are called \emph{loose edges}, while the edges
incident to exactly $1$ vertex are called \emph{half edges}. A \emph{circle} 
of signed graph is the edge set of a simple cycle, a half edge, or a loose
edge. A circle is \emph{positive} or \emph{negative} according to the
product of the signs of its edges. A loose edge is regarded as positive,
while a half edge is regarded as negative. A \emph{handcuff} is a pair of
disjoint circles connected by a path of length $\geq 1$, or two circles who
only share a single vertex. If both circles of a handcuff are positive we
say the handcuff \emph{balanced}, and if both circles are negative we say
the handcuff is \emph{contra-balanced}.

While Zaslavsky introduces two natural matroids associated to a signed
graph, our focus is on the frame matroid, which most faithfully extends the
concepts of graph theory via the signed incidence matrix. The circuits of
the signed graphic frame matroid are classified by the following theorem of
Zaslavsky in \cite{SG}.

\begin{theorem}
\label{S1} $C$ is a circuit of the signed graphic frame matroid $M(\Sigma )$
if, and only if, $C$ is the edge set of a positive circle or a
contra-balanced handcuff in the signed graph $\Sigma $.
\end{theorem}

We can regard a graph as a signed graph in which each edge is positive, so every circle of a graph is necessarily positive, thus theorem \ref{S1} subsumes the graphic circuit classification. Orientations of signed graphs (see \cite{OSG}) motivates the development of incidence-oriented hypergraphs. While our focus is on hypergraphic extensions of balanced signed graphs we also refine the concept of an unbalanced signed graph into balanceable and unbalanceable oriented hypergraphs.

The concept of an incidence-oriented hypergraph extending a signed graph for 
VLSI design and logic synthesis was introduced in 1992 by Shi in \cite{Shi1},
and further developed by Shi and Brzozowski in \cite{ShiBrz}. 
Incidence-oriented hypergraphs and balance were independently developed by
Rusnak in \cite{OHD} as a combinatorial model to extend algebraic and spectral graph theoretic results to integral matrices as well as examine the circuit structure of representable 
matroids --- this paper is an adaptation of the introduction and classification of those balanced minimal dependency results.


\section{An Introduction to Oriented Hypergraphs}
\label{Section2}

\subsection{Introductory Definitions}

Let $V$ and $E$ be disjoint finite sets whose respective elements are called 
\emph{vertices} and \emph{edges}. An \emph{incidence function} is a
function $\iota :V\times E\rightarrow \mathbb{Z}_{\geq 0}$, while a vertex $v$ 
and an edge $e$ are said to be \emph{incident with respect to }$\iota $
if $\iota (v,e)\neq 0$. An \emph{incidence} is a triple $(v,e,k)$ where $v$
and $e$ are incident and $k\in \{1$, $2$, $3$, . . . , $\iota (v,e)\}$. The value 
$\iota (v,e)$ is called the \emph{multiplicity} of the incidence.

Let $\mathcal{I}_{\iota }$ be the set of incidences determined by $\iota $. 
Since the set $\mathcal{I}_{\iota }$ also determines the incidence
function we immediately drop the subscript notation and simply write $\mathcal{I}$. 
An \emph{incidence orientation} is a function $\sigma :\mathcal{I}\rightarrow \{+1,-1\}$. 
Every incidence $(v,e,k)$ is naturally extended to a quadruple $(v,e,k,\sigma (v,e,k))$ 
called an \emph{oriented incidence}. An \emph{oriented hypergraph} is the quadruple 
$(V,E,\mathcal{I},\sigma )$. This formulation of oriented incidence is an extension of orientations of signed graphs in \cite{OSG}.

When drawing oriented hypergraphs the vertices are depicted as points in
the plane while edges will be depicted as shaded regions in the plane whose
incident vertices appear on its boundary. \ An oriented incidence $%
(v,e,k,\sigma (v,e,k))$ is drawn within edge $e$ as an arrow entering $v$ if 
$\sigma (v,e,k)=+1$, or an arrow exiting $v$ if $\sigma (v,e,k)=-1$.

\begin{figure}[h!]
\centering
\includegraphics{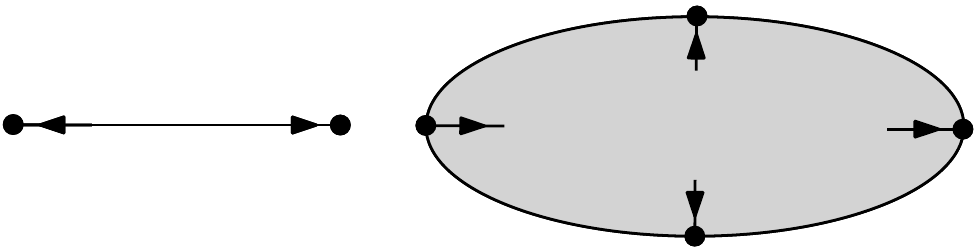}
\caption{Two oriented hyperedges.}
\label{fig:OHEdges}
\end{figure}

A triple $(V,E,\mathcal{I})$ is a \emph{hypergraph}, and all
definitions that do not depend on $\sigma $ will be defined on the
underlying hypergraph of an oriented hypergraph and will be inherited by the
oriented hypergraph.

A hypergraph is \emph{simple} if $\iota (v,e)\leq 1$ for all $v$ and $e$,
and for convenience we will write $(v,e)$ instead of $(v,e,1)$ if $G$ is a
simple hypergraph. Two, not necessarily distinct, vertices $v$ and $w$ are
 \emph{adjacent with respect to edge }$e$ if there exists
incidences $(v,e,k_{1})$ and $(w,e,k_{2})$ such that $(v,e,k_{1})\neq
(w,e,k_{2})$. An \emph{adjacency} is a quintuple $(v,k_{1};w,k_{2};e)$ where 
$v$ and $w$ are adjacent with respect to edge $e$ using incidences $(v,e,k_{1})$ 
and $(w,e,k_{2})$.

The \emph{degree}, or \emph{valency}, of a vertex is equal
to the number of incidences containing that vertex and is denoted $\deg (v)$. 
A vertex whose degree equals $0$ is \emph{isolated} and a vertex whose degree 
equals $1$ is \emph{monovalent}. The \emph{size} of an edge is the number
of incidences containing that edge, and an edge of size $k$ is called a $k$\emph{-edge}.

A \emph{path} is a set of vertices, edges, and incidences of a hypergraph
that form a sequence $
a_{0},i_{1},a_{1},i_{2},a_{2},i_{3},a_{3},...,a_{m-1},i_{m},a_{m}$, where $
\{a_{j}\}$ is an alternating sequence of vertices and edges, $i_{j}$ is an
incidence containing $a_{j-1}$ and $a_{j}$, and\ no vertex, edge, or
incidence is repeated. The first and last elements of this sequence are 
the \emph{end-points} of the path. A path where both end-points are
vertices is a \emph{vertex-path}, a path where both end-points are edges
is an \emph{edge-path}, and a path where one end-point is a vertex and
the other is an edge is a \emph{cross-path} as it \textquotedblleft
crosses\textquotedblright\ the incidence structure from a vertex to an edge.

A hypergraph is \emph{connected} if for any two distinct elements
of $V\cup E$ there exists a path in $G$ containing them. A hypergraph that
is not connected is \emph{disconnected}. A \emph{connected
component} is a maximal connected subhypergraph. An edge whose removal 
increases the number of connected components is an \emph{isthmus}, 
a vertex whose removal increases the number of connected components is 
a \emph{cut vertex}, and an incidence whose removal increases the 
number of connected components is a \emph{shoal}.

A \emph{circle of length }$k$ is a set of $k$ vertices, $k$ edges, and $2k$
incidences that form a sequence $%
a_{0},i_{1},a_{1},i_{2},a_{2},i_{3},a_{3},...,a_{2k-1},i_{2k},a_{2k}$, where 
$\{a_{j}\}$ is an alternating sequence of vertices and edges, $i_{j}$ is an
incidence containing $a_{j-1}$ and $a_{j}$, and\ no vertex, edge, or
incidence is repeated except $a_{0}=a_{2k}$. \ By symmetry we may assume
that $a_{0}$ is a vertex. \ A circle $C$ is \emph{degenerate} if, for some
edge $a_{i}\in C$, there is a vertex $v\in C$ such that $v$ is not $a_{i-1}$
or $a_{i+1}$, and $v$ is incident to $a_{i}$. \ A circle that is not
degenerate is called \emph{pure}.

Given a hypergraph $G$ and a monovalent vertex $v$ of $G$ we say $v$ is a 
\emph{leaf of }$G$ if the edge incident to $v$ is not contained in a circle
of $G$, however, we say $v$ is a \emph{thorn of }$G$ if the edge incident to 
$v$ is contained in some circle of $G$. An edge containing a leaf is 
a \emph{twig} while an edge containing a thorn is a \emph{briar}.

Given an adjacency $(v,k_{1};w,k_{2};e)$ we define the 
\emph{sign of the adjacency} as 
\begin{equation*}
sgn_{e}(v,k_{1};w,k_{2})=-\sigma (v,e,k_{1})\sigma (w,e,k_{2})\text{.}
\end{equation*}%
This is shortened to $sgn_{e}(v,w)=-\sigma (v,e)\sigma (w,e)$ if $G$ is simple. 
If $v$ and $w$ are not adjacent via edge $e$ we say the
sign of the non-adjacency is $0$. It should be noted that signed $2$-edges as discussed in \cite{SG} correspond to an edge with a single signed adjacency for oriented hypergraphs.

If $B=\{a_{0},i_{1},a_{1},i_{2},a_{2},i_{3},a_{3},...,a_{n-1},i_{n},a_{n}\}$ is a
circle or a path, then the \emph{sign of }$B$ is%
\begin{equation*}
sgn(B)=(-1)^{p}\prod_{h=1}^{n}\sigma (i_{h})\text{,}
\end{equation*}%
where%
\begin{equation*}
p=\left\lfloor \tfrac{n}{2}\right\rfloor .
\end{equation*}%
This implies that the sign of a circle is the product of the signs of all adjacencies in the circle.

Given a hypergraph $G=(V_{G},E_{G},\mathcal{I}_{G})$ a \emph{subhypergraph} $%
H$ of $G$ is the hypergraph $H=(V_{H},E_{H},\mathcal{I}_{H})$ where $%
V_{H}\subseteq V_{G}$, $E_{H}\subseteq E_{G}$, and $\mathcal{I}_{H}\subseteq 
\mathcal{I}_{G}\cap \mathcal{(}V_{H}\times E_{H}\times \mathbb{Z} )$. \ This
definition is more relaxed than conventional definitions as it allows for
only parts of edges to appear in the subhypergraph, giving the flexibility
to have incidence-centric treatments of subhypergraphs in addition to the
usual edge-centric and vertex-centric subhypergraphs.

We are often interested in subhypergraphs with more structure then a general
subhypergraph. \ Let $G=(V,E,\mathcal{I})$ be a hypergraph, and let $%
U\subseteq V$ and $F\subseteq E$. \ The \emph{cross-induced subhypergraph of 
}$G$\emph{\ on }$(U,F)$ is the subhypergraph $G$:$(U,F)=(U,F,\mathcal{I}\cap
(U\times F\times \mathbb{Z}))$. \ If $U=V$ we say that the subhypergraph is
an \emph{edge-restriction to }$F$ and write $G|F$. \ An \emph{edge-induced
hypergraph}\textit{\ }is the hypergraph $G$:$F=(W,F,\mathcal{I}\cap (W\times
F\times \mathbb{Z}))$ where $W=\{v\in V:$ $v$ is incident to some $f\in F\}$%
. \ All hypergraphic containment will take place in the edge-induced
ordering unless otherwise stated.

\subsection{The Incidence Matrix}

Given a labeling $v_{1}$, $v_{2}$, $v_{3}$, . . . , $v_{m}$ of the elements
of $V$, and $e_{1}$, $e_{2}$, $e_{3}$, . . . , $e_{n}$ of the elements of $E$%
, of an oriented hypergraph $G$, the \emph{incidence matrix of }$G$ is the $%
m\times n$ matrix $\mathrm{H}_{G}=[\eta _{ij}]$, where 
\begin{equation*}
\eta _{ij}=\sum\limits_{k=1}^{\iota (v_{i},e_{j})}\sigma (v_{i},e_{j},k)%
\text{.}
\end{equation*}

If $G$ is simple, then this is equivalent to%
\begin{equation*}
\eta _{ij}=\left\{ 
\begin{tabular}{ll}
$0$, & if $(v_{i},e_{j})\notin \mathcal{I}\text{,}$ \\ 
$1$, & if $\sigma (v_{i},e_{j})=+1\text{,}$ \\ 
$-1$, & if $\sigma (v_{i},e_{j})=-1\text{.}$%
\end{tabular}%
\right.
\end{equation*}

Every simple oriented hypergraph with a labeled vertex set and\ labeled edge
set has a representation as a $\{0,\pm 1\}$-matrix using its incidence
matrix. Moreover, a $\{0,\pm 1\}$-matrix with labeled columns and rows 
has a unique representation as a simple oriented hypergraph
with edge set equal to the column labels, vertex set equal to the row
labels, and a vertex $v$ and an edge $e$ are incident if the $(v,e)$-entry
in the matrix is non-zero.

Non-simple oriented hypergraphs may have incidence matrix entries other than 
$0$, $+1$, or $-1$, for example, if there are three incidences containing
the same vertex and edge each oriented $+1$, then a value of $+3$ would appear in
the incidence matrix. It is also possible that two incidences at the same
vertex within the same edge could be signed $+1$ and $-1$ and produce a net
value of $0$ in the incidence matrix. To avoid such redundancies all
multiple incidences are regarded as having the same orientation unless
stated otherwise.

An oriented hypergraph is said to be \emph{dependent} if the columns of its
incidence matrix are dependent, and adopt similar conventions for all
matrix related terminology. The classification of the minimal column
dependencies of a $\{0,\pm 1\}$-matrix $\mathrm{H}$ begins with the 
following simple lemma.

\begin{lemma}
\label{L5} If an oriented hypergraph contains a monovalent vertex, then it
is not minimally dependent.
\end{lemma}

\begin{proof}
If an oriented hypergraph contains a monovalent vertex, then there is a row
with a single non-zero entry and the corresponding column cannot belong to a
minimal dependency.
\end{proof}

\subsection{Incidence Duality}

Given a hypergraph $G$ the \emph{incidence dual} $G^{\ast }$ is the
hypergraph obtained by reversing the roles of the vertices and edges. That
is, given an oriented hypergraph $G=(V,E,\mathcal{I},\sigma )$, its
incidence dual\textit{\ }$G^{\ast }$ is the oriented hypergraph $(E,V,%
\mathcal{I}^{\ast },\sigma ^{\ast })$ where $\mathcal{I}^{\ast
}=\{(e,v,k):(v,e,k)\in \mathcal{I}\}$, and $\sigma ^{\ast }:\mathcal{I}%
^{\ast }\rightarrow \{+1,-1\}$ such that $\sigma ^{\ast }(e,v,k)=\sigma
(v,e,k)$. Observe that $\mathcal{I}^{\ast }$ determines an incidence
function $\iota ^{\ast }$ where $\iota ^{\ast }(e,v)=\iota (v,e)$. In
graph theory a line graph can be regarded as the graphical approximation of
incidence duality. A number of algebraic graph theoretic results hold in the more general setting of oriented hypergraphs and incidence duality, see \cite{AH1}.

A number of structures are closed under incidence duality. By interchanging 
the roles of edges and vertices the incidence dual of a path is still a path. 
Specifically, the incidence dual of a vertex-path is an edge-path, the incidence 
dual of an edge-path is a vertex-path, and the incidence dual of a cross-path 
is a cross-path. Similarly, the incidence dual of a circle is still a circle. 
However, we have a better result for circles:

\begin{lemma}
\label{L7} The following are true for a circle $C$ in an oriented hypergraph 
$G$.

\begin{enumerate}
\item $C$ is pure in $G$ if, and only if, $C^{\ast }$ is pure in $G^{\ast }$.

\item The sign of $C$ in $G$ is equal to the sign of $C^{\ast }$ in $G^{\ast
}$.
\end{enumerate}
\end{lemma}

\begin{proof}
Incidence duality reverses vertices and edges, which, by symmetry, does not alter purity
of a circle. Moreover, the incidence signs are also unchanged in the
incidence dual, so the sign of a circle also remains unchanged.
\end{proof}


\section{Operations on Oriented Hypergraphs}
\label{Section3}

\subsection{Deletion, Switching, and 2-Contraction}

\emph{Weak edge-deletion of edge }$e$, denoted $G\smallsetminus e$, is the
hypergraph resulting from the set deletion of the edge $e$ from $E$ along
with the removal of any incidences containing $e$ from $\mathcal{I}$. The
incidence dual of weak edge-deletion is \emph{weak vertex-deletion }and is
denoted $G\smallsetminus v$ for $v\in V$, and removes the vertex $v$ from $V$
along with any incidences containing $v$. The removal of a single edge or a single 
vertex has the following effect on the incidence matrix.

\begin{lemma}
\label{L10} Weak edge-deletion and weak vertex-deletion are equivalent to
column-deletion and row-deletion in the corresponding incidence matrix.
\end{lemma}

Deletion of a vertex along with all incident edges is called \emph{strong
vertex-deletion}, while its incidence dual operation is \emph{strong
edge-deletion}.

A \emph{vertex-switching function} is any function $\theta:V\rightarrow \{-1,+1\}$.  \emph{Vertex-switching} the oriented hypergraph $G$ means replacing $\sigma$ by $\sigma^{\theta}$, defined by: $\sigma^{\theta}(v,e,k_1)=\theta(v)\sigma(v,e,k_1)$; producing the oriented hypergraph $G^{\theta}=(V,E,\mathcal{I},\sigma^{\theta})$.  Vertex-switching produces an adjacency sign $sgn^{\theta}$, defined by: $sgn_{e}^{\theta}(v,k_{1};w,k_{2})= \theta(v)sgn_{e}(v,k_{1};w,k_{2})\theta(w)$. 

\emph{Edge-switching} is the incidence dual of vertex-switching, and negates all incidences that contain a given edge. Observe that switching has the effect of negating a column or row.

\begin{lemma}
\label{L11} Edge-switching and vertex-switching are equivalent to column negation or row 
negation, respectively, in the corresponding incidence matrix.
\end{lemma}

\begin{lemma}
\label{L12}Edge-switching does not alter the signs of any adjacencies in an
oriented hypergraph.
\end{lemma}

\begin{proof}
Consider the adjacency $(v,k_{1};w,k_{2};e)$. Since switching an edge $e$ negates all incidences containing $e$, the sign of this adjacency is
\begin{equation*}
sgn_{e}(v,k_{1};w,k_{2})=-\sigma (v,e,k_{1})\sigma (w,e,k_{2})
\end{equation*}
before switching, and has sign 
\begin{equation*}
-[-\sigma (v,e,k_{1})][-\sigma (w,e,k_{2})]=sgn_{e}(v,k_{1};w,k_{2})
\end{equation*}
after switching.
\end{proof}

\begin{lemma}
\label{L13}Vertex-switching does not alter the signs of any circles in an
oriented hypergraph.
\end{lemma}

\begin{proof}
Let $C=a_{0},i_{1},a_{1},i_{2},a_{2},i_{3},a_{3},...,a_{2k-1},i_{2k},a_{2k}$
be a circle and $a_{j}$ is the vertex we wish to switch. Switching $a_{j}$
will negate incidences $i_{j}$ and $i_{j+1}$, and the switched circle will
have the same sign.
\end{proof}

Switching plays an essential part in defining contraction in an oriented
hypergraph in order to have it agree with matroid contraction in the column
dependency matroid of the incidence matrix. Because we will later restrict
ourselves to a certain family of oriented hypergraphs, we only need to focus
on the contraction of $2$-edges and its incidence dual operation.

The origins of signed $2$\emph{-edge-contraction} appear in \cite{SG} and
its development remains faithful to the corresponding matroidal contraction. 
A positive $2$-edge is contracted as a graphic edge, while a negative $2$-edge 
is contracted by first switching one of the incident vertices so that the edge is 
positive and then contracting the edge.

Incidence dual to $2$-edge-contraction is $2$\emph{-vertex-contraction} and
can performed by taking the incidence dual, contracting the corresponding
edge, and then dualizing again. We say a vertex is \emph{compatibly
oriented (with respect to two of its incidences)} if the product of the two
incidences is negative. Compatible $2$-vertex-contraction has the effect
of combining the two incident edges into a single new edge with the
contracted vertex removed.

\begin{lemma}
\label{L14}Let $G$ be a minimally dependent oriented hypergraph. If $G^{\prime }$ 
is obtained by a $2$-vertex-contraction of $G$, then $G^{\prime}$ is minimally dependent.
\end{lemma}

\begin{proof}
Let $G=(V,E,\mathcal{I},\sigma )$ be a minimally dependent oriented
hypergraph where $V=\{v_{1}$, $v_{2}$, . . . , $v_{m}\}$ and $E=\{e_{1}$, 
$e_{2}$, . . . , $e_{n}\}$. Let $v_{1}$ be the degree-$2$ vertex we wish to
contract, and suppose the edges are labeled so that $v_{1}$ is incident to
edges $e_{1}$ and $e_{2}$. \ We may assume that $v_{1}$ is compatibly
oriented, if it is not we can switch an incident edge since switching does not alter minimal dependencies.

Since $G$ is minimally dependent, solving $\mathrm{H}\mathbf{x=0}$ yields a
fully supported coefficient vector $\mathbf{x}$. This corresponds to the linear system 
\begin{equation*}
\sum\limits_{i=1}^{n}\alpha _{i}\mathbf{e}_{i}=\mathbf{0}\text{,}
\end{equation*}%
where the values $\alpha _{i}$ are the entries of $\mathbf{x}$. Summing
only row $v_{1}$ we see that $\alpha _{1}e_{1,1}+\alpha _{2}e_{1,2}=0$ since 
$\deg (v_{1})=2$. \ Since $v_{1}$ is compatible we know that if $e_{1,1}=\pm
1$, then $e_{1,2}=\mp 1$, so $\alpha _{1}=\alpha _{2}$. Thus columns $\mathbf{e}_{1}$ 
and $\mathbf{e}_{2}$ may be replaced with a single new column 
$\mathbf{e}=\mathbf{e}_{1}+\mathbf{e}_{2}$, and row $v_{1}$ may be deleted as it contains only $0$ entries. The resulting oriented hypergraph remains minimally dependent.
\end{proof}

\subsection{Subdivision and Column Splitting}

The inverse of $2$-vertex-contraction is called \emph{edge-subdivision}. 
In a drawing of an oriented hypergraph, edge-subdivision bipartitions the
incidences of an edge and \textquotedblleft pinches off\textquotedblright\
the edge to produce a new degree-$2$ vertex between two newly created edges. 
A subdivision is \emph{compatible} if the product of the two new
incidences is negative and is \emph{incompatible} if the product of the two
new incidences is positive. If a subdivision is compatible we can
immediately contract the newly introduced vertex to reclaim the original
oriented hypergraph.

\begin{figure}[h!]
\centering
\includegraphics{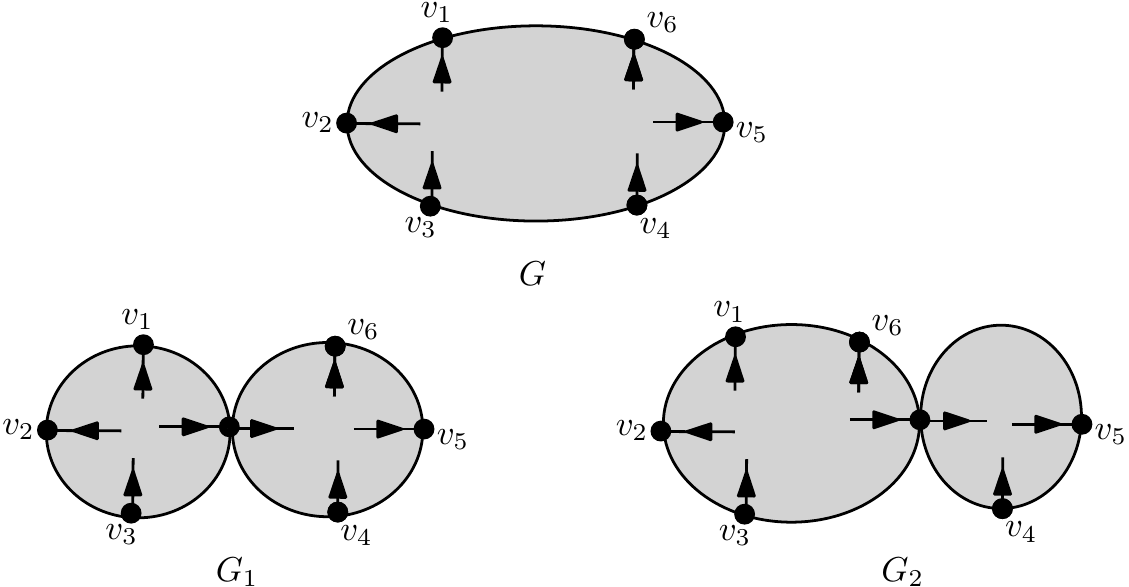}
\caption{Two different subdivisions of a hyperedge.}
\label{fig:OHSubdiv}
\end{figure}

Compatible subdivision plays a central role in understanding the structure
of dependencies for two reasons. \ First, compatible subdivision does not
alter the signs of any existing circles. \ Second, compatible subdivision
does not alter minimal dependencies.

\begin{lemma}
\label{L16}The sign of a path between two vertices in any compatible
subdivision of an edge $e$ is equal to the sign of their adjacency in $e$.
\end{lemma}

\begin{proof}
Let $v$ and $w$ be two vertices incident to edge $e$, and edge $e$ is to be subdivided into
edges $e_{1}$ and $e_{2}$. If $v$ and $w$ and in the same side of the
bipartition of a subdivision of $e$ they will have the same adjacency sign.
If $v$ and $w$ are in different parts of the bipartition of $e$ the newly
introduced vertex $u$ between them is compatibly oriented and the sign of
the resulting $vw$-path is 
\begin{align*}
\lbrack -\sigma (v,e_{1})\sigma (u,e_{1})][-\sigma (u,e_{2})\sigma
(w,e_{2})]& =[\sigma (u,e_{1})\sigma (u,e_{2})][\sigma (v,e_{1})\sigma
(w,e_{2})] \\
& =-[\sigma (v,e_{1})\sigma (w,e_{2})] \\
& =-[\sigma (v,e)\sigma (w,e)]\text{.}
\end{align*}
Which is the same as the original adjacency sign in $e$.
\end{proof}

From this we immediately have the following corollary.

\begin{corollary}
\label{L17}Compatible subdivision does not change the signs of any circles.
\end{corollary}

The operation of subdivision has an effect on the incidence matrix called 
\emph{column splitting}. As with subdivision, we have \emph{compatible} and 
\emph{incompatible column splitting} depending on whether the associated
subdivision was compatible or incompatible.

\begin{lemma}
\label{L18}Let $M$ be an $m\times n$ $\{0,\pm 1\}$-matrix and $M^{\prime }$
be a matrix obtained by compatible column splitting $M$. \ If $M$ is
minimally dependent, then so is $M^{\prime }$.
\end{lemma}

\begin{proof}
Since the columns of $M$ are minimally dependent there is a single solution
(up to scaling) of the matrix equation $M\mathbf{x=0}$. Moreover, by
minimality, the vector $\mathbf{x}$ satisfying this equation must have full
support as no smaller supported vector can produce a dependency. \ Writing
this as a linear combination of the column vectors we have 
\begin{equation*}
\sum\limits_{i=1}^{n}\alpha _{i}\mathbf{c}_{i}=\mathbf{0}\text{,}
\end{equation*}%
where the $\alpha _{i}$'s are the entries of $\mathbf{x}$ and $\mathbf{c}_{i}
$ is the column vector corresponding to column $i$.

Let $\mathbf{c}_{1}$ be the column split into the new columns $\mathbf{d}%
_{1}^{\prime }$ and $\mathbf{d}_{2}^{\prime }$, and the new row created be $%
\mathbf{r}$. \ We can assume that $\mathbf{r}$ is introduced as the last row
of the newly formed matrix $M^{\prime }$. \ Let $\mathbf{d}_{1}$ and $%
\mathbf{d}_{2}$ be the column vectors obtained by removing row $\mathbf{r}$
from $\mathbf{d}_{1}^{\prime }$ and $\mathbf{d}_{2}^{\prime }$ respectively.

Extend each column $\mathbf{c}_{i}$, $i\geq 2$, to a new column $\mathbf{c}%
_{i}^{\prime }$ where%
\begin{equation*}
\mathbf{c}_{i}^{\prime }=\left[ 
\begin{array}{c}
\mathbf{c}_{i} \\ 
0%
\end{array}%
\right]
\end{equation*}%
and the entry $0$ appears in row $\mathbf{r}$. \ Thus the matrix $M^{\prime
} $ obtained by this compatible column splitting is%
\begin{eqnarray*}
M^{\prime } &=&\left[ 
\begin{tabular}{l|l|l|l|l|l}
$\mathbf{d}_{1}^{\prime }$ & $\mathbf{d}_{2}^{\prime }$ & $\mathbf{c}%
_{2}^{\prime }$ & $\mathbf{c}_{3}^{\prime }$ & $\text{. . . }$ & $\mathbf{c}%
_{n}^{\prime }$%
\end{tabular}%
\right] \\
&=&\left[ 
\begin{tabular}{l|l|l|l|l|l}
$\mathbf{d}_{1}$ & $\mathbf{d}_{2}$ & $\mathbf{c}_{2}$ & $\mathbf{c}_{3}$ & $%
\text{. . . }$ & $\mathbf{c}_{n}$ \\ 
$\pm 1$ & $\mp 1$ & $0$ & $0$ & $\text{. . . }$ & $0$%
\end{tabular}%
\right]
\end{eqnarray*}%
where the last row is row $\mathbf{r}$.

Taking the same $\alpha _{i}$'s, $i\geq 2$, that determined the minimal
dependency for $M$, we let the coefficients on both $\mathbf{d}_{1}^{\prime
} $ and $\mathbf{d}_{2}^{\prime }$ be $\alpha _{1}$. \ Taking this linear
combination of columns of $M^{\prime }$ gives us 
\begin{equation*}
\sum\limits_{i=2}^{n}\alpha _{i}\mathbf{c}_{i}^{\prime }+\alpha _{1}\mathbf{d%
}_{1}^{\prime }+\alpha _{1}\mathbf{d}_{2}^{\prime }\text{.}
\end{equation*}%
\ For rows $\mathbf{r}_{1}$ through $\mathbf{r}_{m}$ we necessarily get $0$
as this corresponds to the original linear combination of columns except
that we use either $\mathbf{d}_{1}$ or $\mathbf{d}_{2}$ depending on which
column supports the entry of $\mathbf{c}_{1}$. \ However, for row $\mathbf{r}
$ the linear system gives $\pm \alpha _{1}\mp \alpha _{1}=0$, so this linear
combination forms a dependency.

The dependency is clearly minimal as both new columns are required in the
dependency, and $\alpha _{1}\neq 0$ since the original matrix was minimally
dependent and no $\alpha _{i}$ is zero.
\end{proof}

\begin{corollary}
\label{L19}If $G$ is minimally dependent and $G^{\prime }$ is a compatible
subdivision of $G$, then $G^{\prime }$ is minimally dependent.
\end{corollary}


\section{Circle and Path Analogs}
\label{Section4}

\subsection{Inseparability and Flowers}

So far we have only translated the simple, closed path, property of a graphic circle 
to oriented hypergraphs. We now extend another property of graphic circles: the
property of being minimally inseparable.

An oriented hypergraph is \emph{inseparable} if every pair of incidences is
contained in a circle. An inseparable oriented hypergraph is \emph{circle-covered} 
if it contains a circle, or is a single $0$-edge.

\begin{lemma}
\label{L22}A circle-covered hypergraph contains no monovalent vertices, $1$-edges, 
or isolated vertices.
\end{lemma}

A \emph{flower} is a circle-covered oriented hypergraph that is minimal in
the edge-induced subhypergraphic ordering. Clearly every graphic circle is a flower,
however, there are additional flowers in oriented hypergraphs other than
circles.

\begin{figure}[h!]
\centering
\includegraphics{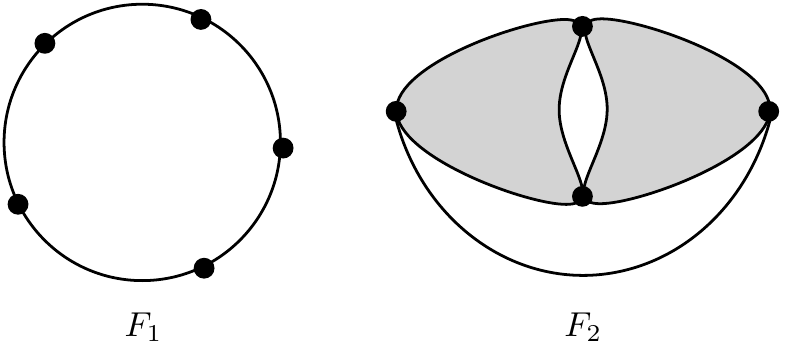}
\caption{Two examples of flowers.}
\label{fig:flower}
\end{figure}

While flowers of an oriented hypergraph can have many circles, the concept
of a flower is simplified in signed graphs.

\begin{proposition}
\label{L24}$F$ is a flower of a signed graph if, and only if, $F$ is a circle or a loose edge.
\end{proposition}

A \emph{pseudo-flower} is an oriented hypergraph containing at least one thorn 
such that the weak-deletion of all thorns results in a flower. The subhypergraph 
resulting from the weak-deletion of thorns in a pseudo-flower is called its \emph{flower-part}.
 A $k$\emph{-pseudo-flower} is a pseudo-flower with exactly $k$ thorns.

\begin{figure}[h!]
\centering
\includegraphics{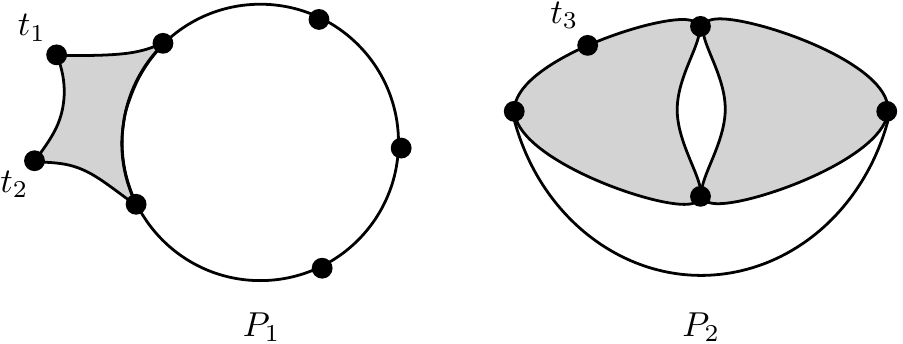}
\caption{Two pseudo-flowers that contain the hypergraphs from Figure \protect
\ref{fig:flower} as flower-parts.}
\label{fig:pflower}
\end{figure}

Pseudo-flowers occur only as a degenerate example in signed graphs.

\begin{proposition}
\label{L25}$P$ is a pseudo-flower of a signed graph if, and only if, $P$ is a half edge.
\end{proposition}

\subsection{Arteries}

An \emph{artery} is a connected, circle-free, $1$-edge-free hypergraph in
which the degree of every vertex is $1$ or $2$, or is a single vertex. The
divalent vertices of an artery are called \emph{internal vertices} of the
artery, while the non-divalent vertices are called the \emph{external
vertices} of the artery. A $k$\emph{-artery} is an
artery with exactly $k$ external vertices.

\begin{figure}[h!]
\centering
\includegraphics{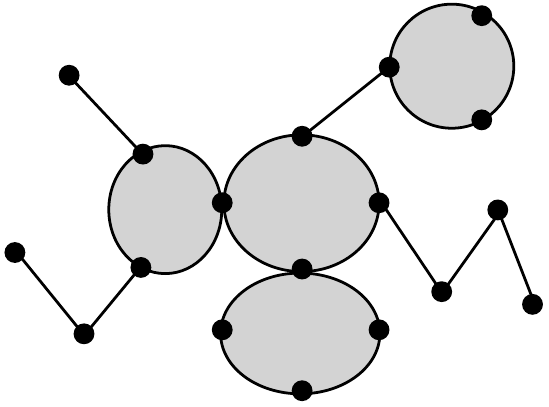}
\caption{An artery.}
\label{fig:artery}
\end{figure}

The concept of an artery lies somewhere between a graphic tree and a path. \
An artery must also contain a unique a path between every pair of its
vertices since it is incidence dual to a graphic tree, however, every
internal vertex must also have degree equal to $2$. \ The simplest arteries
are a single vertex, which is a $1$-artery or \emph{vertex-artery}, and a
graphic path, which is a $2$-artery.

Just as a graphical path can be thought of as a subdivision of a $2$-edge, a 
$k$-artery can be regarded as a subdivision of a $k$-edge. \ The structure
of arteries can be characterized through the operation of subdivision as
indicated by the following useful lemmas.

\begin{lemma}
\label{L31}$A$ is a $k$-artery, $k\geq 2$, if, and only if, $A$ can be vertex-contracted
 into a $k$-edge.
\end{lemma}

\begin{lemma}
\label{L32}$A$ is a $k$-artery, $k\geq 2$, if, and only if, $A$ is a subdivision of a $k$-edge.
\end{lemma}

Flowers and pseudo-flowers are simplified in signed graphs, as indicated by
Propositions \ref{L24} and \ref{L25}, and we have the following result for
arteries of signed graphs.

\begin{proposition}
\label{L33}$A$ is an artery of a signed graph if, and only if, $A$ is a path.
\end{proposition}

\subsection{Arterial Connections and Hypercircles}

Two hypergraphs that are either disjoint or have a single vertex in common
are said to be \emph{nearly-disjoint}. \ An \emph{arterial connection of
hypergraphs} is the union of a collection of pairwise nearly-disjoint
hypergraphs\ $\mathcal{H}$ with a collection of pairwise disjoint arteries $
\mathcal{A}$ satisfying:

\begin{enumerate}
\item[AC1.] The arterial connection is connected.

\item[AC2.] If $H\in \mathcal{H}$, $A\in \mathcal{A}$, and $H\cap A\neq
\emptyset $, then $H\cap A=(v,\emptyset ,\emptyset )$ and $v$ is an external
vertex of $A$.

\item[AC3.] If $H_{1}$, $H_{2}\in \mathcal{H}$ and $H_{1}\cap H_{2}\neq
\emptyset $, then $H_{1}\cap H_{2}=(w,\emptyset ,\emptyset )\in \mathcal{A}$.

\item[AC4.] If $H_{1}$, $H_{2}\in \mathcal{H}$ and $H_{1}\cap
H_{2}=(w,\emptyset ,\emptyset )$, then $H_{1}$ and $H_{2}$ are the only
elements of $\mathcal{H}$ that contain $w$.

\item[AC5.] Weak-deletion of any edge or vertex of an artery in $\mathcal{A}$
disconnects the arterial connection.
\end{enumerate}

An arterial connection of special interest is a \emph{thorn-connection}
which is the union of a collection of nearly-disjoint pseudo-flowers $
\mathcal{P}$ and a collection of pairwise disjoint arteries $\mathcal{A}$
satisfying:

\begin{enumerate}
\item[TC1.] A thorn-connection is an arterial connection.

\item[TC2.] If $P\in \mathcal{P}$, $A\in \mathcal{A}$, and $P\cap A\neq
\emptyset $, then $P\cap A=(t,\emptyset ,\emptyset )$ where $t$ is a thorn
of $P$.
\end{enumerate}

Observe that if two pseudo-flowers of a thorn-connection share a vertex in
common then it must be a thorn of each.

An arterial connection is said to be \emph{floral} if it contains no
monovalent vertices. The oriented hypergraph resulting from the 
vertex-contraction of the vertices belonging to the arteries of a 
floral thorn-connection is a \emph{hypercircle}. Observe that this contraction 
preserves the flower-parts of each pseudo-flower. A hypercircle containing 
exactly $k\geq 2$ flower-parts is called a $k$\emph{-hypercircle}, while 
a $0$-edge is a $0$\emph{-hypercircle}, and flower is a $1$\emph{-hypercircle}.

We say two pseudo-flowers $P_{1}$ and $P_{2}$ are \emph{adjacent} if they
share a single briar in common which is also an isthmus in $P_{1}\cup P_{2}$. 
A hypercircle is the vertex-contraction of a floral thorn-connection
into adjacent pseudo-flowers.


\section{The Cyclomatic Number and Incidence}
\label{Section5}

\subsection{The Incidence Graph}

The \emph{oriented incidence graph of an oriented hypergraph} $G=(V_{G},E_{G},%
\mathcal{I}_{G},\sigma )$ is the oriented bipartite graph $\Gamma _{G}$ with
vertex set $V_{\Gamma }=V_{G}\cup E_{G}$, edge set $E_{\Gamma }=\mathcal{I}%
_{G}$ and orientation function $\sigma $. Observe that $\Gamma _{G}$
contains no parallel edges if, and only if, $G$ is a simple hypergraph.

\begin{figure}[h!]
\centering
\includegraphics{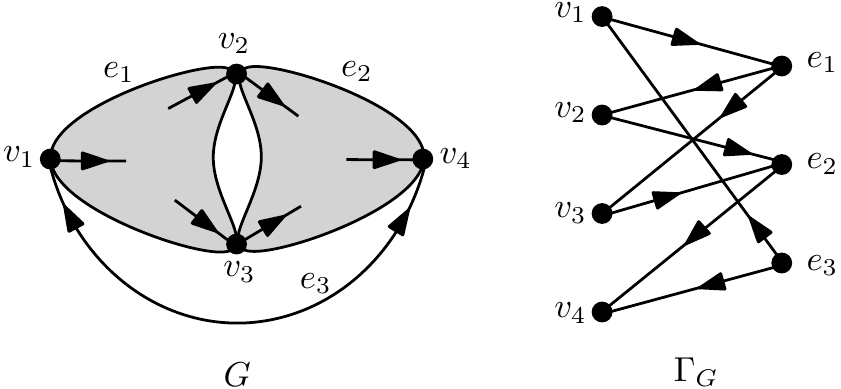}
\caption{An oriented hypergraph $G$ and its incidence graph $\Gamma $.}
\label{fig:incidencegraph}
\end{figure}

The incidence graph provides an alternate point of view to examine some
oriented hypergraphic concepts.

\begin{lemma}
\label{L34}$C$ is a circle of an oriented hypergraph $G$ if, and only if, $C$
is a circle of the incidence graph $\Gamma $.
\end{lemma}

\begin{proof}
A hypergraphic circle is a sequence  $a_{0},i_{1},a_{1},i_{2},a_{2},i_{3},a_{3},...,a_{2k-1},i_{2k},a_{2k}$, where 
$\{a_{j}\}$ is an alternating sequence of vertices and edges, $i_{j}$ is an
incidence containing $a_{j-1}$ and $a_{j}$, and\ no vertex, edge, or
incidence is repeated except $a_{0}=a_{2k}$. In the incidence graph this is
an alternating sequence of vertices (hypergraph vertices and edges) and
edges (hypergraph incidences) where no vertex or edge repeats and whose
end-points coincide, which is a graphic circle. The definitions coincide
when translating between hypergraphs to incidence graphs.
\end{proof}

A \emph{chord} of a graphic circle $C$ is an edge not in $C$ whose end-points are in $C$.

\begin{lemma}
\label{L35}$C$ is a degenerate circle of an oriented hypergraph $G$ if, and
only if, there exists a chord of $C$ in the incidence graph $\Gamma $.
\end{lemma}

\begin{proof}
A circle $C$ is degenerate if there is an incidence that does not appear in the
circle but belongs to an edge and vertex of the circle. In the incidence
graph this incidence is an edge not in $C$ whose end-points are in $C$. The
concepts coincide when translating between hypergraphs to incidence graphs.
\end{proof}

We immediately have a restatement of Lemma \ref{L35} in terms of pure
circles.

\begin{corollary}
\label{L36}$C$ is a pure circle of an oriented hypergraph $G$ if, and only
if, $C$ is chord-free in the incidence graph $\Gamma $.
\end{corollary}

Some of the terminology introduced for oriented hypergraphs are direct
translations from graphic definitions when viewed through incidence graphs.

\begin{lemma}
\label{L37}Let $G$ be an oriented hypergraph with incidence graph $\Gamma $.
\ $G$ is inseparable if, and only if, $\Gamma $ is inseparable.
\end{lemma}

\begin{proof}
$G$ is inseparable if every pair of incidences is contained in a circle,
while $\Gamma $ is inseparable if every pair of edges is contained in a
circle. \ The edges of $\Gamma $ are the incidences of $G$.
\end{proof}

\begin{lemma}
\label{L38}Let $G$ be an oriented hypergraph with incidence graph $\Gamma $.
\ $G$ is simple if, and only if, $\Gamma $ is simple.
\end{lemma}

\begin{proof}
$G$ is simple if there are no multiple incidences, while $\Gamma $ is simple
there are no parallel edges, as loops do not occur in bipartite graphs. \
The edges of $\Gamma $ are the incidences of $G$.
\end{proof}

\subsection{The Cyclomatic Number for Oriented Hypergraphs}

Since circles of an oriented hypergraph are in one-to-one correspondence
with the circles in its incidence graph the graphic cyclomatic number of
the incidence graph can be regarded as the cyclomatic number for the 
oriented hypergraph. The graphic cyclomatic number $\varphi $ is equal to
the number of edges that lie outside a maximal forest of a graph $\Gamma $ 
and is given by the following equation.
\begin{equation*}
\varphi _{\Gamma }=\left\vert E_{\Gamma }\right\vert -\left\vert V_{\Gamma
}\right\vert +c\text{,}
\end{equation*}
where $c$ is the number of connected components of $\Gamma $.

Using the graphic cyclomatic number for the bipartite incidence graph of an
oriented hypergraph, where $V_{\Gamma }$ consists of the vertices and edges
of $G$ and $E_{\Gamma }$ consists of the incidences of $G$, we define the 
\emph{cyclomatic number of an oriented hypergraph }$G$ as%
\begin{equation*}
\varphi _{G}:=\left\vert \mathcal{I}_{G}\right\vert -(\left\vert
V_{G}\right\vert +\left\vert E_{G}\right\vert )+c\text{,}
\end{equation*}%
where $c$ is the number of connected components of $G$.

We also have the following alternate, incidence dual, ways of calculating
the hypergraphic cyclomatic number which are consistent with Berge's
formulation of the cyclomatic number in \cite{Berge2}:%
\begin{equation*}
\varphi _{G}=\sum\limits_{e\in E_{G}}\left\vert e\right\vert -(\left\vert
V_{G}\right\vert +\left\vert E_{G}\right\vert )+c\text{,}
\end{equation*}%
or%
\begin{equation*}
\varphi _{G}=\sum\limits_{v\in V_{G}}\left[ \deg (v)\right] -(\left\vert
V_{G}\right\vert +\left\vert E_{G}\right\vert )+c\text{.}
\end{equation*}

\begin{lemma}
\label{L39}The hypergraphic cyclomatic number of a graph is equal to its
graphic cyclomatic number.
\end{lemma}

\begin{proof}
In a graph $\Gamma $ every edge has size $2$ so $\left\vert \mathcal{I}%
\right\vert =2\left\vert E\right\vert $. \ Replacing $\left\vert \mathcal{I}%
\right\vert $ in the hypergraphic cyclomatic number with $2\left\vert
E\right\vert $ yields the result.
\end{proof}

Since the two cyclomatic numbers agree on graphs we will refer to a single
cyclomatic number, the oriented hypergraphic version, and translate existing
results for the graphic cyclomatic number to the hypergraphic cyclomatic
number.

An alternate way to interpret the graphic cyclomatic number is that it is
the minimal number of circles that must be \textquotedblleft
broken\textquotedblright\ in order to be left with an acyclic graph.  While this is
normally accomplished by deleting edges from the graph, in the incidence graph 
the edges are incidences of the corresponding oriented hypergraph.  \emph{Breaking} in a hypergraph is the operation defined by deleting a single element from the incidence set $\mathcal{I}$, while \emph{breaking a circle} is the deletion of a single incidence of that circle.

\begin{corollary}
\label{L40}If $G$ is a hypergraph, then the cyclomatic number $\varphi _{G}$
is the minimal number of circles of $G$ that need to be broken to yield an
acyclic hypergraph.
\end{corollary}

A minimal collection of circles whose breaking leaves an acyclic hypergraph
is called a collection of \emph{essential circles}. The concept of an
essential circle of a hypergraph is similar to that of a fundamental circle
of a graph. A fundamental circle arises from the graphical property that a
unique circle is created when introducing an edge outside of a spanning
forest, while an essential circle is a hypergraphic property where a unique
circle is created when introducing an incidence outside a hypergraph
corresponding to a spanning forest in the incidence graph. We adopt the
term \textquotedblleft essential circle\textquotedblright\ as the incidence-centric 
oriented hypergraphic concept and reserve the word \textquotedblleft
fundamental\textquotedblright\ as an edge-centric concept. Specifically,
the choice of terminology is motivated as to not create confusion with the
matroid theoretic concept of a fundamental circuit.

\begin{corollary}
\label{L41}If $G$ is a hypergraph, then $\varphi _{G}$ is equal to the size
of any collection of essential circles in $G$.
\end{corollary}

\begin{lemma}
\label{L42}If $G$ is an oriented hypergraph and $H$ a subdivision of $G$,
then $\varphi _{G}=\varphi _{H}$.
\end{lemma}

\begin{proof}
Subdivision cannot create any new connected components and creates exactly
one new edge, one new vertex, and two new incidences, producing a net change
of $0$ in the cyclomatic number.
\end{proof}

Subdivision may create new circles but does not destroy existing circles, so subdividing may only create new collections of essential circles.

\begin{corollary}
\label{L43}Any collection of essential circles of\ an oriented hypergraph $G$
corresponds to a collection of essential circles in a subdivision $H$ of $G$.
\end{corollary}

\subsection{Theta Graphs}

We now examine a configuration in oriented hypergraphs with specific signed
circle properties. A \emph{theta graph} is a set of three internally disjoint
paths whose end-points coincide. A \emph{vertex-theta-graph} is a theta graph
whose end-points are vertices, an \emph{edge-theta-graph} is a theta graph
whose end-points are edges, and a \emph{cross-theta-graph} is a theta graph
whose end-points consist of one vertex and one edge.

\begin{figure}[h!]
\centering
\includegraphics{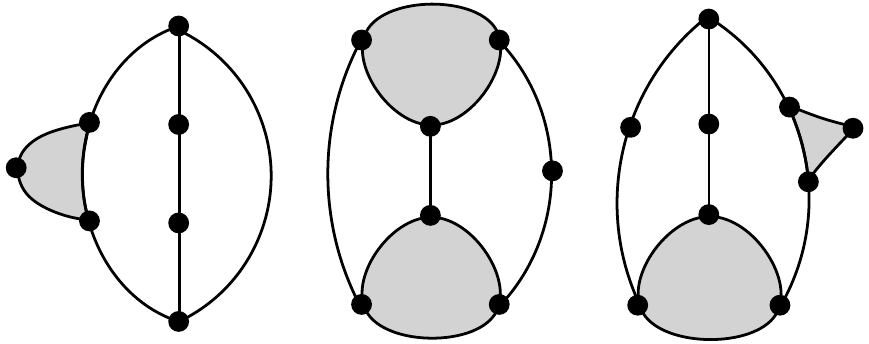}
\caption{Hypergraphs that contain a vertex-theta, edge-theta, and
cross-theta, respectively.}
\label{fig:thetas}
\end{figure}

The paths of a theta graph form three internally disjoint paths in the
incidence graph. The paths of a vertex-theta begin and end on the vertex
side of the incidence graph, the paths of an edge-theta begin and end on the
edge side of the incidence graph, and the paths of a cross-theta have one
end-point on each side of the incidence graph. Since the incidence graph is
bipartite the paths in vertex-thetas and edge-thetas must have even length
in the incidence graph, while the paths of a cross-theta must have odd
length in the incidence graph.

\begin{lemma}
\label{L44}If an oriented hypergraph contains an incidence of multiplicity $%
k\geq 3$, then it contains a cross-theta.
\end{lemma}

\begin{proof}
Any incidence with multiplicity $k\geq 3$ in an oriented hypergraph
corresponds to $k$ parallel edges in the incidence graph. \ Any three of
them correspond to a cross-theta in the oriented hypergraph.
\end{proof}

\begin{lemma}
\label{L45}If an oriented hypergraph contains a degenerate circle, then it
contains a cross-theta.
\end{lemma}

\begin{corollary}
\label{L46}Every circle in a cross-theta-free oriented hypergraph is pure.
\end{corollary}

Oriented hypergraphs that contain cross-thetas provide the most
significant obstacle in the classification of the minimal dependencies of an
oriented hypergraph. Moreover, cross-thetas persist structurally under
incidence duality and subdivision.

\begin{lemma}
If $G$ contains a cross-theta, then the incidence dual $G^{\ast }$ contains
a cross-theta.
\end{lemma}

\begin{lemma}
\label{L47}If $G$ contains a cross-theta, then any subdivision of $G$
contains a cross-theta.
\end{lemma}

While incidences of multiplicity $3$ or greater contain cross-thetas, we may
subdivide the oriented hypergraph to produce a simple,
degenerate-circle-free, oriented hypergraph that necessarily contains a
cross-theta. The relationship between cross-thetas, degenerate circles,
and multiple incidences plays an essential role in extending the
investigation of minimal column dependencies of $\{0,\pm 1\}$-matrices to
determining minimal column dependencies of integral matrices.

Not only do circles factor prominently into the classification of the
minimal dependencies of signed graphs, but the sign of each circle plays an
important part as well. As a result, we turn our attention to the signed
circle structure of theta graphs.

\begin{lemma}
\label{L48}A vertex-theta or an edge-theta contains an even number of
negative circles.
\end{lemma}

\begin{proof}
We will show the result for vertex-thetas, and observe that incidence
duality and Lemma \ref{L7} completes the proof for edge-thetas.

Let the paths connecting the end-points of a vertex-theta be $P_{1}$, $P_{2}$, and 
$P_{3}$. Let the signs of the three paths connecting the end-points be $\varepsilon _{1}$, $\varepsilon _{2}$, and $\varepsilon_{3}$ respectively.

From these three paths we have the following circles: $C_{1}=P_{1}\cup
P_{2}$, $C_{2}=P_{1}\cup P_{3}$, and $C_{3}=P_{2}\cup P_{3}$. The sign of $C_{1}$ 
is $\varepsilon _{1}\varepsilon _{2}$, the sign of $C_{2}$ is $\varepsilon _{1}\varepsilon _{3}$, and the sign of $C_{3}$ is $\varepsilon _{2}\varepsilon _{3}$. If the $\varepsilon _{i}$ are all $+$ or all $-$ then each circle is positive, thus there are no negative circles. 
If the $\varepsilon _{i}$ do not all have the same sign then without loss of generality
suppose $\varepsilon _{1}=+$ and $\varepsilon _{2}=-$. Then the sign of $C_{1}$ is negative, and depending on $\varepsilon _{3}$ exactly one of $C_{2} $ or $C_{3}$ will be negative as well.
\end{proof}

\begin{lemma}
\label{L49}A cross-theta contains an odd number of negative circles.
\end{lemma}

\begin{proof}
Let the end-points of the three cross-paths in a cross-theta be vertex $v$ and edge $e$. Let the cross-paths be $P_{1}$, $P_{2}$, and $P_{3}$ and the vertices of $e$ that belong to these cross-paths be $v_{1}$,$v_{2}$, and $v_{3}$, respectively.  Let $C_{ij}$ be the circle formed by paths $P_{i}$ and $P_{j}$ along with
edge $e$.

\textit{Case 1a:} Suppose $sgn(P_{1})=sgn(P_{2})=sgn(P_{3})$, and $\sigma
(v_{1},e)=\sigma (v_{2},e)=\sigma (v_{3},e)$. \ Then $%
sgn(C_{12})=sgn(C_{13})=sgn(C_{23})=-1$.

\textit{Case 1b:} Suppose $sgn(P_{1})=sgn(P_{2})=sgn(P_{3})$, and $\sigma
(v_{1},e)=\sigma (v_{2},e)\neq \sigma (v_{3},e)$. \ Then $sgn(C_{12})=-1$,
and $sgn(C_{13})=sgn(C_{23})=+1$.

\textit{Case 2a:} Suppose $sgn(P_{1})=sgn(P_{2})\neq sgn(P_{3})$, and $%
\sigma (v_{1},e)=\sigma (v_{2},e)=\sigma (v_{3},e)$. \ Then $sgn(C_{12})=-1$%
, and $sgn(C_{13})=sgn(C_{23})=+1$.

\textit{Case 2b:} Suppose $sgn(P_{1})=sgn(P_{2})\neq sgn(P_{3})$, and $%
\sigma (v_{1},e)=\sigma (v_{2},e)\neq \sigma (v_{3},e)$. \ Then $%
sgn(C_{12})=-1$, and $sgn(C_{13})=sgn(C_{23})=+1$.

\textit{Case 2c:} Suppose $sgn(P_{1})=sgn(P_{2})\neq sgn(P_{3})$, and $%
\sigma (v_{1},e)\neq \sigma (v_{2},e)=\sigma (v_{3},e)$. \ Then $%
sgn(C_{13})=-1$, and $sgn(C_{12})=sgn(C_{23})=+1$.

Up to relabeling, these cases exhaust all possible combinations of path
signs and incidence signs in a cross-theta.  In every case there are an odd number 
of negative circles.
\end{proof}

A cross-theta provides a hypergraphic object which must contain a negative
circle regardless of incidence orientation. \ While a cross-theta presents a
problem unique to hypergraphs, the following theorem examines the structural
properties of cross-thetas in flowers.

\begin{theorem}
\label{L50}If a flower contains a vertex of degree $\geq 3$, then it
contains a cross-theta.
\end{theorem}

\begin{proof}
By subdividing out all degenerate circles and all multiple incidences we
only need to consider simple, degenerate-circle-free, flowers since
subdivision does not remove cross-thetas by Lemma \ref{L47}.

Let $F$ be a simple, degenerate-circle-free, flower containing a vertex $v$
such that $\deg (v)\geq 3$, and let three of the edges incident to $v$ be $%
e_{1}$, $e_{2}$, and $e_{3}$. \ Since $F$ is a flower we know that there
must be a circle $C$ containing the incidences $(v,e_{1})$ and $(v,e_{2})$.
\ Also, since $F$ is degenerate-circle-free, $e_{3}$ cannot belong to $C$. \ 
$C$ must contain an edge of size $\geq 3$ in $F$ or there would be a smaller
flower, namely the circle-hypergraph corresponding to $C$, contradicting
minimality of $F$.

Let $\mathcal{P}$ be the collection of paths in $F$ containing the incidence 
$(v,e_{3})$ with one end-point vertex $v$ and the other an element of $C$
such that each path is internally disjoint from $C$. The elements of $%
\mathcal{P}$ are either cross-paths or vertex-paths, depending on non-$v$
end-point of $C$. $\mathcal{P}$ must contain at least one
cross-path or else every path of $\mathcal{P}$ would be a vertex-path and
the hypergraph resulting from the deletion of all the non-end-point elements of
the paths of $\mathcal{P}$ would result in a smaller flower, contradicting
the minimality of $F$.

Let $Q$ be a shortest cross-path of $\mathcal{P}$, and let the end-points of $Q$
be $v$ and $e$. Regard $C$ as two internally disjoint cross-paths $P_{1}$
and $P_{2}$ each with end-points $v$ and $e$ as well, this can be done since $e$
must also be an element of $C$. $P_{1}$, $P_{2}$, and $Q$ are three
internally disjoint cross-paths whose end-points coincide, and $F$ contains a
cross-theta.
\end{proof}

From this result we have the following corollary concerning cross-theta-free
flowers.

\begin{corollary}
\label{L51}Every vertex of a cross-theta-free flower must have degree equal
to $2$.
\end{corollary}

\begin{proof}
We know from Theorem \ref{L50} if a flower has a vertex of degree $\geq 3$
it must contain a cross-theta, so the degree of every vertex in a
cross-theta-free flower must be $\leq 2$. However, a flower cannot contain
a monovalent vertex or an isolated vertex since it would not be inseparable. 
Thus the degree of every vertex in a cross-theta-free flower must be
exactly $2$.
\end{proof}

\subsection{Ear Decompositions of Flowers}

Let $H$ be a hypergraph, $P$ be a path containing at least one incidence,
and $G(P)$ be the path-hypergraph corresponding to the elements of $P$. The
hypergraph $H\cup G(P)$ is said to result from \emph{adjoining an ear to }$H$
if $H\cap G(P)$ consists of only the end-points of $P$. This concept of
adjoining an ear follows the development in \cite{AGT}. When viewing this
process from the incidence graph, adjoining an ear to a hypergraph $H$ is
equivalent to the graphical concept of adjoining an ear in the incidence
graph $\Gamma _{H}$.

Adjoining an ear to a bipartite graph either connects the vertices within a
single part of the bipartition or connects the vertices across the
bipartition with a path. The connecting path is a vertex-path or an
edge-path in the corresponding oriented hypergraph if the end-points lie in a
single part of the bipartition, and is a cross-path if the end-points lie in
different parts of the bipartition. Observe that any path that connects to
the edge-part of the bipartition would increase the size of the edge in the
oriented hypergraph. A hypergraph that can be constructed starting from a
single vertex or edge by sequentially adjoining ears is said to have an 
\emph{ear decomposition}.

The following is a known result (see \cite{AGT}) concerning the structure of graphs.

\begin{theorem}
A connected graph has an ear decomposition if, and only if, it is
inseparable.
\end{theorem}

We are especially interested in the structure of cross-theta-free flowers
and applying this result to the incidence graph of a flower.

\begin{theorem}
\label{L52}If $F$ is a cross-theta-free flower, then every ear decomposition
of $F$ can be regarded as consisting of only edge-paths.
\end{theorem}

\begin{proof}
From Lemma \ref{L37} we know that $\Gamma _{F}$ is inseparable since $F$ is
a flower. \ Given an ear decomposition of $\Gamma _{F}$ regard the first
circle as adjoining an ear to a vertex of $\Gamma _{F}$ belonging to the
edge-part of the vertices. \ This can be done because $\Gamma _{F}$ is
bipartite.

By Theorem \ref{L50} we know that $F$ cannot contain a vertex of degree $3$
or greater, so adjoining additional ears in $\Gamma _{F}$ must connect two
vertices in the edge-part of the\ bipartition or else we would have a degree-%
$3$ vertex.
\end{proof}

\begin{corollary}
\label{L53}Let $\mathcal{P}$ be a collection of paths of an ear
decomposition of a cross-theta-free flower $F$. \ Every path of $\mathcal{P}$
must contain a unique vertex that does not belong to any other path of $%
\mathcal{P}$.
\end{corollary}

\begin{proof}
Using Theorem \ref{L52} we see that any path in an ear decomposition of $%
\Gamma _{F}$ as must connect two vertices of $\Gamma _{F}$ that are edges of 
$F$. \ Since $\Gamma _{F}$ is bipartite every such path must have even
length, and every path must contain a unique vertex of\ $\Gamma _{F}$ that
corresponds to a vertex of $F$.
\end{proof}

This provides us with the following property for a collection of essential
circles of a cross-theta-free oriented hypergraph.

\begin{corollary}
\label{L54}Given a collection of essential circles of a cross-theta-free
flower $F$, there exists a set\ of distinct vertex representatives for each
essential circle.
\end{corollary}

\begin{proof}
Let $F$ be a cross-theta-free flower and $\mathcal{C}$ be a set of essential
circles of $F$. \ Since $F$ is inseparable, by Theorem \ref{L52}, $F$ can be
built by adjoining edge-path ears. \ Moreover, a collection of essential
circles has a natural ear decomposition since they generate all the circles
of $F$. \ By Corollary \ref{L53} there must exist a vertex in each essential
circle that does not belong to any other circle in $\mathcal{C}$.
\end{proof}


\section{The Notion of Balance for Oriented Hypergraphs}
\label{Section6}

\subsection{Variations of Balance}

We say an oriented hypergraph is \emph{balanced} if all circles are
positive. An oriented hypergraph is \emph{balanceable} if there are
incidences that can be negated so that the resulting oriented hypergraph is
balanced. An oriented hypergraph that is not balanceable is said to be 
\emph{unbalanceable}. Clearly, any oriented hypergraph containing a
cross-theta must necessarily be unbalanceable by Theorem \ref{L49}. In
fact, we will see that cross-thetas are the only obstruction to
balanceability by translating existing formulations of balance to oriented
hypergraphs.

The concept of a balanced non-oriented hypergraph was introduced by Berge
in \cite{Berge1} as one of a number of different generalizations of
bipartite graphs. Berge defined a hypergraph as \emph{balanced} if every odd 
circle has an edge containing three vertices of the circle. In terms of oriented
hypergraphs this is equivalent to all odd circles being degenerate and all
even circles being pure. Moreover, Berge's work can be regarded as
incidence matrices whose entries consist of $0$ and $1$, so every adjacency
is necessarily negative if considered as an oriented hypergraph. If every adjacency 
in an oriented hypergraph is negative, then a circle is negative if, and only if, 
it has odd length.

A balanced $\{0,\pm 1\}$-matrix was introduced by Truemper in \cite{TrAlpha}
as a generalization of a balanced hypergraph. A $\{0,\pm 1\}$-matrix is a 
\emph{hole matrix} if it contains two non-zero entries per row and per
column and no proper submatrix has this property. A hole matrix is \emph{even} 
if the sum of its entries is congruent to $0$ mod $4$, and 
\emph{odd} if the sum of its entries is congruent to $2$ mod $4$. 
A $\{0,\pm 1\}$-matrix $A$ is \emph{balanced} if no submatrix of $A$ is an
odd hole.

There are a number of simple observations translating concepts from balanced
matrices to oriented hypergraphs.

\begin{proposition}
Let $H$ be a hole matrix and $C$ be the corresponding circle in the
associated oriented hypergraph. $\ H$ is a hole submatrix if, and only if, $
C $ is pure.
\end{proposition}

\begin{proposition}
Let $H$ be a hole matrix and $C$ be the corresponding circle in the
associated oriented hypergraph. $H$ is even if, and only if, $C$ is
positive.
\end{proposition}

\begin{proposition}
\label{L55}A $\{0,\pm 1\}$-matrix is balanced if, and only if, every pure
circle in its associated oriented hypergraph is positive.
\end{proposition}

As Proposition \ref{L55} indicates, the difference between the concept of a
balanced matrix and the concept of a balanced oriented hypergraph is one of
purity. A $\{0,\pm 1\}$-matrix is balanced if, and only if, every pure
circle is positive, while an oriented hypergraph is balanced if, and only
if, all circles are positive. This change simply moves degenerate circles
from being thought of as balanceable for matrices to being thought of as
unbalanceable for oriented hypergraphs. The concept of balanceability in an
oriented hypergraph is weakening of the concept of a balanceable matrix, and
a survey of balanced matrices by M. Conforti, G. Cornu\'{e}jols, and K. Vu%
\v{s}kovic can be found in \cite{BM}.

\subsection{Obstructions to Balanceability}

The characterization of the minimal obstructions to balanceability of $\{0,\pm
1\}$-matrices is due to Truemper \cite{TrLog}, while the following
adaptation of Truemper's result to oriented hypergraphs follows \cite{BM,
CGK}.

A \emph{hole} in a graph is a chord-free circle of length $4$ or greater,
while a \emph{wheel} is a subgraph consisting of a hole $H$ and a vertex $v$
having at least three neighbors in $H$. A wheel is \emph{odd} if the
number of neighbors of $v$ in $H$ is odd. A $3$\emph{-path configuration}
in a graph is a subgraph consisting of three internally disjoint paths
between two non-adjacent vertices, and a $3$\emph{-odd-path configuration}
is a $3$-path configuration where each path has odd length. Observe that
in a bipartite graph a $3$-odd-path configuration connects two vertices in
opposite sides of the bipartition using $3$ internally-disjoint paths.

The following characterization of balanceability due to Truemper.

\begin{theorem}
\label{L56}A bipartite graph is balanceable if, and only if, it does not
contain an odd wheel or a $3$-odd-path configuration as a subgraph.
\end{theorem}

If we take the bipartite representation graph of a $\{0,\pm 1\}$-matrix,
then odd wheels and $3$-odd-path configurations are the minimal bipartite
graphs whose corresponding $\{0,\pm 1\}$-matrix must contain an odd hole
matrix.

Truemper's minimal obstructions to balanceability for bipartite graphs can
be translated to oriented hypergraphs since the edges of the bipartite
incidence graph $\Gamma _{G}$ correspond to the incidences of the oriented
hypergraph $G$. Any $3$-odd-path configuration in $\Gamma _{G}$ is a
cross-theta in $G$, moreover, since no path of a $3$-odd-path configuration
consists of a single edge, a $3$-odd-path configuration must correspond to a
non-degenerate circle cross-theta of $G$. The inclusion of degenerate
circle cross-thetas in unbalanceable oriented hypergraphs yields all
cross-thetas as an obstruction to balanceability, which simply relaxes the
non-adjacency requirement of a $3$-path configuration.

Truemper's other minimal obstruction is an odd wheel in $\Gamma _{G}$, which
must contain a cross-theta containing the central vertex of the wheel. 
With the inclusion of degenerate circle cross-thetas, all cross-thetas are
already obstructions to balanceability in oriented hypergraphs, and we have
the following theorem:

\begin{theorem}
\label{L60}An oriented hypergraph $G$ is balanceable if, and only if, it
does not contain a cross-theta.
\end{theorem}

\begin{corollary}
The multiplicity of any incidence in a balanceable oriented hypergraph is at
most $2$.
\end{corollary}

We have already seen in Lemma \ref{L49} that a cross-theta must contain a
negative circle regardless of its incidence orientations. By developing a
theory of balance for oriented hypergraphs used specifically as a refinement
of being negative-circle-free, degenerate circles are not treated separate from 
other cross-thetas. This adaptation allows us to translate Truemper's work
to see that cross-thetas are the only obstruction to balanceability in
oriented hypergraphs, and the investigation into the minimal dependencies
of oriented hypergraphs has a natural division into three categories:
balanced, balanceable, and unbalanceable.


\section{The Circuit Classification of Balanced Oriented Hypergraphs}
\label{Section7}

\subsection{Balanced Flowers}

The classification of the minimal dependencies of graphs is a well known
result.

\begin{theorem}
\label{L61}The minimal dependencies of a graph are circles.
\end{theorem}

Using Proposition \ref{L24}, we can translate Theorem \ref{L61} using
oriented hypergraphic terminology that is indicative of the dependency
results we will obtain for balanced oriented hypergraphs.

\begin{theorem}
The minimal dependencies of a graph are balanced flowers.
\end{theorem}

The focus in this section is on the extension of Theorem \ref{L61} to
oriented hypergraphs by examining balanced flowers. 

\begin{lemma}
\label{L62}A balanced flower does not contain a vertex of degree $\geq 3$.
\end{lemma}

\begin{proof}
If a balanced flower had a vertex of degree $\geq 3$ then by Theorem \ref{L50} 
it would contain a cross-theta, and by Lemma \ref{L49} would contain a negative circle.
\end{proof}

Lemma \ref{L62} incorporates Theorem \ref{L50}\ into the theory of balanced
oriented hypergraphs, and has an immediate result paralleling Corollary \ref%
{L51}.

\begin{lemma}
\label{L63}The degree of every vertex in a balanced flower must be $2$.
\end{lemma}

\begin{proof}
From Corollary \ref{L51} we know that every vertex of a cross-theta-free
flower must have degree equal to $2$. \ A balanced flower is necessarily
cross-theta-free.
\end{proof}

Observe that Lemma \ref{L63}\ implies that the incidence dual of a balanced
flower is a signed graph.

Finally, we arrive at our first family of minimally dependent oriented
hypergraphs.

\begin{theorem}
\label{L64}A balanced flower is minimally dependent.
\end{theorem}

\begin{proof}
Let $F$ be a balanced flower, and observe that a flower cannot contain a $1$-edge.

\textit{Case 1:} \ If $F$ is a $0$-edge, then it corresponds to a single
column of $0$'s and is minimally dependent.

\textit{Case 2:} \ If $F$ consists of only edges of size $2$, then it is
already minimally dependent as it is a positive signed graphic circle.

\textit{Case 3:} \ If $F$ contains an edge of size $3$ or greater, then we
will show that it is minimally dependent.

Let $\mathrm{H}_{F}$ be the incidence matrix of $F$ and let $\Gamma _{F}$ be
the incidence graph of $F$. \ Take a spanning tree of $\Gamma _{F}$ to
determine a collection of essential circles of $F$ by translating the
fundamental circles of $\Gamma _{F}$ to $F$. \ The sign of each essential
circle is $+$ since $F$ is balanced.

Since each essential circle is pure and positive we can take a linear
combination of the rows corresponding to the vertices of that circle to zero
out any row in the square submatrix corresponding to the vertices and edges
of the circle. \ Moreover, we know that every vertex has degree $2$ by Lemma %
\ref{L63}, thus there are no other non-zero entries in the rows of $\mathrm{H%
}_{F}$ outside of the square submatrix corresponding to the circle, and the
entire row in $\mathrm{H}_{F}$ must be zero after row reducing.

Corollary \ref{L54} tells us there is a unique vertex for each essential
circle that does not belong to any other essential circle in the given
collection. For each essential circle take a linear combination of the
rows corresponding to the vertices of that circle so that a row
corresponding to a vertex unique to that essential circle is zero. Since
this vertex is not contained in any other essential circles we can zero out
a row for each essential circle. Thus we can zero out exactly $\varphi
_{F} $ rows since the essential circles of $F$ are fundamental circles of $%
\Gamma _{F}$. Since $\mathrm{H}_{F}$ has $\left\vert V_{F}\right\vert $
rows and we can zero out exactly $\varphi _{F}$ of them to see that the row
rank of $\mathrm{H}_{F}$ is $\left\vert V_{F}\right\vert -\varphi _{F}$.

For $F$ to be minimally dependent the nullity of $\mathrm{H}_{F}$ must
necessarily be $1$. If the nullity of $\mathrm{H}_{F}$ is $1$, then $F$ is
minimally dependent since the weak deletion of any non-empty subset of edges
of $F$ would result in a monovalent vertex since $F$ is minimally
circle-covered, and would not be minimally dependent by Lemma \ref{L5}. 
Since no edge-induced subhypergraph is minimally dependent and the nullity
of $F$ is $1$, then $F$ must be minimally dependent.

In order to complete the proof we must show that $\mathrm{H}_{F}$ has
nullity $1$. Since we know the row rank is $\left\vert V_{F}\right\vert
-\varphi _{F}$ we must show
\begin{equation*}
\left\vert V_{F}\right\vert -\varphi _{F}=\left\vert E_{F}\right\vert -1%
\text{.}
\end{equation*}
Solving for $\varphi _{F}$ this is equivalent to showing 
\begin{equation*}
\varphi _{F}=\left\vert V_{F}\right\vert -\left\vert E_{F}\right\vert +1%
\text{.}
\end{equation*}

By the definition of the cyclomatic number we have 
\begin{equation*}
\varphi _{F}=\left\vert \mathcal{I}_{F}\right\vert -(\left\vert
V_{F}\right\vert +\left\vert E_{F}\right\vert )+1\text{.}
\end{equation*}%
However, since the degree of every vertex of $F$ is equal to $2$ we have $
\left\vert \mathcal{I}_{F}\right\vert =2\left\vert V_{F}\right\vert $. \
Replacing this into the cyclomatic number we get%
\begin{eqnarray*}
\varphi _{F} &=&\left\vert \mathcal{I}_{F}\right\vert -(\left\vert
V_{F}\right\vert +\left\vert E_{F}\right\vert )+1 \\
&=&2\left\vert V_{F}\right\vert -(\left\vert V_{F}\right\vert +\left\vert
E_{F}\right\vert )+1 \\
&=&\left\vert V_{F}\right\vert -\left\vert E_{F}\right\vert +1\text{.}
\end{eqnarray*}

Solving this for $\left\vert V_{F}\right\vert -\varphi _{F}$ we get 
\begin{equation*}
\left\vert V_{F}\right\vert -\varphi _{F}=\left\vert E_{F}\right\vert -1
\text{,}
\end{equation*}%
and the nullity of $\mathrm{H}_{F}$ is equal to $1$.
\end{proof}

Note that Theorem \ref{L64} can be proved using signed graph theory since the incidence dual of a balanced oriented hypergraph is a signed graph.

\subsection{Balanced Pseudo-Flowers}

A pseudo-flower is a result from abstracting to hypergraphs. Since
balanced flowers are minimally dependent, we examine the balanced
flower-parts for a similar simplification locally within each pseudo-flower.

There is only a single signed graphic example of a balanced minimal
dependency involving pseudo-flowers: two $1$-edges connected by a path of
length $\geq 0$. Each $1$-edge is a pseudo-flower whose flower-part is the 
$0$-edge resulting from weak deletion of the vertex. It is natural to ask if an 
oriented hypergraph consisting of a single $k$-edge (or $k$-artery via subdivision) 
with a $1$-edge at each vertex is minimally dependent. Clearly it is, since there 
is a single column containing no zeroes and exactly one column for each row such 
that a linear combination of $1$-edge-columns yields the $k$-edge-column.

Oriented hypergraphs, however, can have more pseudo-flowers than just $1$-edge 
pseudo-flowers.

\begin{lemma}
\label{L65}Let $P_{1}$, $P_{2}$, . . . , $P_{k}$ be a collection of $k$,
disjoint, balanced\ $1$-pseudo-flowers, and $e$ a $k$-edge that meets only
the thorn of each $P_{i}$. \ The oriented hypergraph $G:=P_{1}\cup P_{2}$ $%
\cup $ . . . $\cup $ $P_{k}\cup e$ is minimally dependent.
\end{lemma}

\begin{proof}
Each pseudo-flower\ $P_{i}$ contains a balanced flower-part, thus the degree
of every vertex must be equal to $2$. \ For each $P_{i}$ there are $\varphi
_{P_{i}}$ essential circles so there are $\varphi _{P_{i}}$ rows that can be
zeroed out since the flower-part is balanced, and must contain only vertices
of degree equal to $2$.

Since there are no circles in $G$ other than those in the flower parts of
each $P_{i}$ we have 
\begin{eqnarray*}
\varphi _{G} &=&\sum\limits_{i=1}^{k}\varphi _{P_{i}} \\
&=&\left\vert \mathcal{I}_{G}\right\vert -(\left\vert V_{G}\right\vert
+\left\vert E_{G}\right\vert )+1\text{.}
\end{eqnarray*}

Since the degree of every vertex in $G$ is $2$ we have $\left\vert \mathcal{I%
}_{G}\right\vert =2\left\vert V_{G}\right\vert $. \ Substituting into $%
\varphi _{G}$ we get%
\begin{eqnarray*}
\varphi _{G} &=&\left\vert \mathcal{I}_{G}\right\vert -(\left\vert
V_{G}\right\vert +\left\vert E_{G}\right\vert )+1 \\
&=&2\left\vert V_{G}\right\vert -(\left\vert V_{G}\right\vert +\left\vert
E_{G}\right\vert )+1 \\
&=&\left\vert V_{G}\right\vert -\left\vert E_{G}\right\vert +1\text{.}
\end{eqnarray*}%
Solving for $\left\vert V_{G}\right\vert -\varphi _{F}$ we get 
\begin{equation*}
\left\vert V_{G}\right\vert -\varphi _{F}=\left\vert E_{G}\right\vert -1%
\text{.}
\end{equation*}

That is, the row rank of the incidence matrix of $G$ is $\left\vert
E_{G}\right\vert -1$, so $G$ is dependent with nullity equal to $1$.

To see that $G$ is minimally dependent observe that the weak deletion of any
non-empty subset of edges would either disconnect $G$ or result in a
monovalent vertex, and is not minimally dependent. Since no edge-induced 
subhypergraph is of $G$ is minimally dependent, $G$ must be minimally dependent.
\end{proof}

The thorns of each pseudo-flower in Lemma \ref{L65} can be switched so that
every thorn is compatibly oriented with respect to edge $e$. Vertex-contracting 
these thorns produces a collection of $k$ adjacent pseudo-flowers sharing a single 
common isthmus. Moreover, this is obtained by operations that do not alter minimal 
dependencies, giving us the following corollary:

\begin{corollary}
\label{L66}A balanced hypercircle with a single isthmus is minimally
dependent.
\end{corollary}

A cautionary note concerning Corollary \ref{L66}: vertex-contracting a
thorn-connection into a hypercircle may not be possible in the larger
ambient oriented hypergraph, so any comparisons between thorn-connections
and hypercircles must be done on the edge-induced subhypergraph in order to
examine the structure of minimal dependencies. In other words, we must restrict 
to a specific set of columns when searching for dependency.

It is important to point out that the subdivision of an isthmus in a
balanced hypercircle does not need to be compatible to preserve the minimal
dependency of that hypercircle since any newly created vertex will not
belong to any circle in the subdivision. It could, however, alter another
minimal dependency in a larger ambient oriented hypergraph in which the
subdivided edge belongs to a circle.

A subdivision of $G$ is \emph{balanced} if the subdivision is compatible, or
the subdivision is incompatible and the newly created vertex does not belong
to a circle in the subdivision of $G$.

\begin{lemma}
\label{L67}A subdivision $H$ of $G$ is balanced if, and only if, the circles
of $H$ corresponding to circles of $G$ have the same sign in both $G$ and $H$.
\end{lemma}

\begin{lemma}
\label{L68}If $G$ is minimally dependent, then any balanced subdivision of $G$ 
is minimally dependent.
\end{lemma}

\begin{proof}
Let $H$ be a balanced subdivision of $G$. Call the subdivided edge $e$ and
the new edges resulting from the subdivision $e_{1}$ and $e_{2}$.

If the balanced subdivision is compatible then, from Corollary \ref{L19}, we
know that a compatible subdivision of a minimal dependency is still
minimally dependent. 

If the balanced subdivision is incompatible then the newly created vertex in the 
subdivision does not belong to any circle of $H$. Since the new vertex does not 
belong to any circle of $H$ it does not increase the cyclomatic number or the nullity, 
but does increase the number of vertices and edges each by $1$, so $H$ must be dependent.

To see that $H$ is minimally dependent observe that both $H\smallsetminus
e_{1}$ and $H\smallsetminus e_{2}$ contains a monovalent vertex, and by
Lemma \ref{L5}, neither can be minimally dependent. Moreover, weak
deletion of any other edge is equivalent to weak deletion is $G$, which is
already minimally dependent. So no proper edge-induced subhypergraph of $H$ is
minimally dependent, but $H$ is dependent, so $H$ is minimally dependent.
\end{proof}

\begin{corollary}
\label{L69}A floral thorn-connection of $k$, balanced,\ $1$-pseudo-flowers
is minimally dependent.
\end{corollary}

Building on this we have the following lemma concerning balanced hypercircles:

\begin{theorem}
\label{L69a}A balanced hypercircle is minimally dependent.
\end{theorem}

\begin{proof}
Let $\mathcal{H}$ be a balanced hypercircle with maximal pseudo-flowers $%
P_{1}$, $P_{2}$, . . . , $P_{k}$, whose respective flower-parts are $F_{1}$, 
$F_{2}$, . . . , $F_{k}$. \ Since their flower-parts are pairwise disjoint
and $\mathcal{H}$ is balanced, every vertex belonging to some flower-part
must have degree equal to $2$. \ Also note that every thorn of a
pseudo-flower belongs to the flower-part of another\ adjacent pseudo-flower,
so all the vertices of $\mathcal{H}$ must have degree equal $2$, thus giving 
$\left\vert \mathcal{I}_{\mathcal{H}}\right\vert =2\left\vert V_{\mathcal{H}%
}\right\vert $.

Observe that any collection of essential circles of $\mathcal{H}$ is the
union of a collection of essential circles for each flower-part since the
flower-parts are pairwise disjoint, and there are no circles outside of the
flower-parts. \ Thus we have that%
\begin{equation*}
\varphi _{\mathcal{H}}=\sum\limits_{i=1}^{k}\varphi
_{P_{i}}=\sum\limits_{i=1}^{k}\varphi _{F_{i}}\text{.}
\end{equation*}

For each flower-part we are able to zero out $\varphi _{F_{i}}$ rows for a
total of $\varphi _{\mathcal{H}}$ zero-rows. \ However, 
\begin{equation*}
\varphi _{\mathcal{H}}=\left\vert \mathcal{I}_{\mathcal{H}}\right\vert
-(\left\vert V_{\mathcal{H}}\right\vert +\left\vert E_{\mathcal{H}%
}\right\vert )+1\text{,}
\end{equation*}%
and substituting $\left\vert \mathcal{I}_{\mathcal{H}}\right\vert
=2\left\vert V_{\mathcal{H}}\right\vert $ we have%
\begin{eqnarray*}
\varphi _{\mathcal{H}} &=&2\left\vert V_{\mathcal{H}}\right\vert
-(\left\vert V_{\mathcal{H}}\right\vert +\left\vert E_{\mathcal{H}%
}\right\vert )+1 \\
&=&\left\vert V_{\mathcal{H}}\right\vert -\left\vert E_{\mathcal{H}%
}\right\vert +1\text{.}
\end{eqnarray*}

Solving this for $\left\vert V_{\mathcal{H}}\right\vert -\varphi _{\mathcal{H%
}}$ we see that%
\begin{equation*}
\left\vert V_{\mathcal{H}}\right\vert -\varphi _{\mathcal{H}}=\left\vert E_{%
\mathcal{H}}\right\vert -1\text{.}
\end{equation*}

Thus the row rank is one less than the number of columns, and the incidence
matrix has nullity equal to $1$. The weak deletion of any non-empty set of edges leaves a monovalent vertex, which is not minimally dependent by Lemma \ref{L5}. So no edge-induced 
subhypergraph of $\mathcal{H}$ is minimally dependent, and $\mathcal{H}$ is dependent, so $\mathcal{H}$ is minimally dependent.
\end{proof}

Since flowers are $1$-hypercircles, and $0$-edges are $0$-hypercircles, we
can regard every balanced minimal dependency discussed so far as
subdivisions of balanced hypercircles. \ These, in fact, are the only
balanced minimal dependencies.

\begin{theorem}
\label{L70}$G$ is a balanced minimal dependency if, and only if, $G$ is a
balanced subdivision of a balanced hypercircle.
\end{theorem}

\begin{proof}
Theorem \ref{L69a} and Lemma \ref{L68} tells us that a balanced subdivision
of a balanced hypercircle is minimally dependent, so all that is left to see
is the converse.

To see the converse let $G$ be a balanced minimal dependency, and observe
that $G$ is a connected oriented hypergraph that cannot contain any vertices
of degree equal to $0$ or $1$, by Lemma \ref{L5}. \ Since the degree of any
vertex in a minimal dependency is at least $2$, $G$ must be a $0$-edge, a
floral thorn-connection of $1$-edge pseudo-flowers, or it must contain a
circle. Clearly, a $0$-edge is minimally dependent, and is a $0$%
-hypercircle by definition, while a floral thorn-connections of $1$-edge
pseudo-flowers are minimal dependencies by Corollary \ref{L69}, so we must
show that if $G$ contains a circle, then it is a balanced subdivision of a
hypercircle.

Suppose $G$ is a balanced minimal dependency that contains a circle. \ Since 
$G$ contains a circle it must contain some circle belonging to a flower or
flower-part of a pseudo-flower. \ If $G$ is a flower, then it is minimally
dependent by Theorem \ref{L64}. \ Moreover, $G$ cannot properly contain a
flower without violating minimality of the dependency.

If $G$ contains a circle and is not a flower, then $G$ must contain a
maximal pseudo-flower $P_{0}$. \ Since $G$ is balanced, every vertex in the
flower-part of $P_{0}$ must have degree equal to $2$ in $P_{0}$, by Lemma %
\ref{L63}. \ Thus, the degree of each vertex of $G$ is at least $2$, and
thorns of $P_{0}$ must connect to the rest of $G$ via some edge in $%
G\smallsetminus P_{0}$. \ Since there are no monovalent vertices in $G$, the
paths leading from the thorns of $P_{0}$ into $G\smallsetminus P_{0}$ must
reach a circle or terminate at a $1$-edge. \ If there are no other circles,
then $G$ consists of a single pseudo-flower $P_{0}$ and disjoint paths
leaving each thorn that meet $1$-edge pseudo-flowers. \ Observe that the
paths leaving the thorns of $P_{0}$ must begin with an anchor of an artery
or else it violates minimality of the dependency.

If there is a circle of $G$ not in $P_{0}$, then there must exist a circle whose distance 
from $P_{0}$ is minimal. Let $P_{1}$ be a maximal pseudo-flower containing a nearest circle.
 Observe that $P_{0}$ and $P_{1}$ are flower-part-disjoint or there would exist a circle 
containing elements from each flower-parts, producing a larger pseudo-flower and 
contradicting maximality of the pseudo-flowers. Thus, there is a single path between 
$P_{0}$ and $P_{1}$; it is possible that $P_{0}$ and $P_{1}$ are adjacent. \
Let $A_{1}$ be the largest artery containing the path between $P_{0}$ and $%
P_{1}$ that avoids all circles of $G$. \ If all the circles of $G$ belong to
the flower-parts of $P_{0}$ and $P_{1}$, then $P_{0}\cup P_{1}\cup A_{1}$
must connect to $1$-edge pseudo-flowers at the unused anchors of $A_{1}$,
producing a minimal dependency.

If there is a circle of $G$ that does not belong to $P_{0}\cup P_{1}\cup
A_{1}$, there must exist another maximal pseudo-flower $P_{2}$ flower-part
disjoint from $P_{0}$ and $P_{1}$ whose distance from $P_{0}\cup P_{1}\cup
A_{1}$ is minimal. Let $A_{2}$ be the largest artery containing the path between $P_{2}$ and $P_{0}\cup P_{1}\cup A_{1}$ that internally avoids $P_{2}$ and $P_{0}\cup P_{1}\cup A_{1}$ and all circles of $G$. If all the circles of $G$ belong to the flower-parts of $P_{0}$, $%
P_{1}$, and $P_{2}$, then $P_{0}\cup (P_{1}\cup A_{1})\cup (P_{2}\cup A_{2})$
must connect to $1$-edge pseudo-flowers at the unused anchors of $A_{1}$ and 
$A_{2}$, producing a minimal dependency.

If there is a circle of $G$ does not belong to $P_{0}\cup (P_{1}\cup
A_{1})\cup (P_{2}\cup A_{2})$ we inductively add maximal pseudo-flowers and
arteries nearest to, and avoiding, the previous collection except with the possibility of
monovalent vertices. We form the union 
\begin{equation*}
P_{0}\cup (P_{1}\cup A_{1})\cup \text{ . . . }\cup (P_{k}\cup A_{k})
\end{equation*}
until all circles are exhausted. \ Since there are no
circles outside the pseudo-flowers, the remaining monovalent vertices must
be incident to a single $1$-edge to form the minimal dependency. \ This
forces $G$ to be a subdivision of a balanced hypercircle.
\end{proof}



\end{document}